\documentclass[12pt,reqno]{amsart}
\usepackage{amsmath}
\usepackage{amssymb}
\usepackage{amstext}
\usepackage{a4wide}
\usepackage{graphicx}
\usepackage{enumerate}
\allowdisplaybreaks \numberwithin{equation}{section}
\usepackage{color}
\usepackage{cases}

\numberwithin{equation}{section}

\newtheorem{theorem}{Theorem}[section]
\newtheorem{lemma}{Lemma}[section]

\newtheorem{proposition}{Proposition}[section]
\newtheorem{remark}{Remark}[section]

\newtheorem{definition}{Definition}[section]

\begin{document}

\title[Global existence of weak solutions to the Navier-Stokes equations]{Global existence of weak solutions to the compressible Navier-Stokes equations with temperature-depending viscosity coefficients}

\author{Guodong Wang, Bijun Zuo}

\address{Institute for Advanced Study in Mathematics, Harbin Institute of Technology, Harbin 150001, P.R. China}
\email{wangguodong@hit.edu.cn}
\address{College of Mathematical Sciences, Harbin Engineering University, Harbin {\rm150001}, PR China}
\email{bjzuo@amss.ac.cn}

 %\keywords{}

%\thanks{}

\begin{abstract}
This paper is devoted to the global existence of weak solutions to the three-dimensional compressible Navier-Stokes equations with heat-conducting effects in a bounded domain.
The viscosity and the heat conductivity coefficients are assumed to be functions of the temperature, and the shear viscosity coefficient may vanish as the temperature goes to zero.
The proof is to apply Galerkin method to a suitable approximate system with several parameters and obtain uniform estimates for the approximate solutions. The key ingredient in obtaining the required estimates is to apply De Giorgi's iteration to the modified temperature equation, from which we can get a lower bound for the temperature not depending on the artificial viscosity coefficient introduced in the modified momentum equation, which makes the compactness argument available as the artificial viscous term vanishes.
\end{abstract}

\maketitle
\section{Introduction and Main Result}
\setcounter{equation}{0}
In this paper, we study the global existence of weak solutions to the following three-dimensional compressible Navier-Stokes equations with temperature effects:
\begin{equation}\label{1.1}
\begin{cases}
\partial_t\varrho+{\rm div}(\varrho\mathbf{u})=0,\\
\partial_t(\varrho\mathbf{u})+{\rm div}(\varrho\mathbf{u}\otimes\mathbf{u})+\nabla p={\rm div}\mathbb{S},\\
\partial_t(\varrho\vartheta)+{\rm div}(\varrho\mathbf{u}\vartheta)+{\rm div}\mathbf{q}=\mathbb{S}:\nabla\mathbf{u}-\vartheta p_{\vartheta}(\varrho){\rm div}\mathbf{u}.
\end{cases}
\end{equation}
Here $\varrho=\varrho(t,x)$ is the density of the fluid, $\mathbf{u}=\mathbf{u}(t,x)$ is the velocity field, $\vartheta=\vartheta(t,x)$ is the temperature, $p=p(t,x)$ is the pressure determined by the constitutive equation
\begin{equation}\label{1.2}
p=p(\varrho,\vartheta)=p_{e}(\varrho)+\vartheta p_{\vartheta}(\varrho),
\end{equation}
with $p_{e}$ being the elastic pressure and $\vartheta p_{\vartheta}$ being the thermal pressure,
$\mathbf{q}$ denotes the heat flux of the fluid satisfying the Fourier's law
\begin{equation*}
\mathbf{q}=-\kappa(\vartheta)\nabla\vartheta,
\end{equation*}
with $\kappa=\kappa(\vartheta)>0$ being the heat conductivity coefficient depending on the fluid temperature,
and $\mathbb{S}$ denotes the viscous stress tensor 
\begin{equation*}
\mathbb{S}=\mu(\nabla\mathbf{u}+\nabla^T\mathbf{u})+\lambda{\rm div}\mathbf{u}\mathbb{I},
\end{equation*}
where $\mu=\mu(\vartheta)$ and $\lambda=\lambda(\vartheta)$ are the shear and bulk viscosity coefficients respectively depending on the temperature.    We assume that
\begin{equation}\label{1.002}
\mu(\vartheta)\geq 0,\quad\lambda(\vartheta)+\frac{2}{3}\mu(\vartheta)\geq0.
\end{equation}
Note that $\mu(\vartheta)$ is allowed to degenerate in the region of absolutely zero temperature. 

We will consider system \eqref{1.1} in a smooth bounded domain $\Omega\subset \mathbb{R}^3$, and impose the following initial and boundary conditions
\begin{equation}\label{1.4}
(\varrho,\varrho\mathbf{u},\vartheta)(0,x)=(\varrho_0,\mathbf{m}_0, \vartheta_0)(x) \quad{\rm in}\,\,\Omega.
\end{equation}
\begin{equation}\label{1.3}
\mathbf{u}(t,x)=0,\quad\nabla\vartheta(t,x)\cdot\mathbf{n}(x)=0\quad{\rm on} \,\,{[0,T]\times\partial\Omega},
\end{equation}
where $\mathbf{n}(x)$ is the unit outward normal vector to the boundary at
$x\in\partial\Omega$.

There have been a huge number of works in the literature concerning the global existence of solutions to the compressible Navier-Stokes equations. In particular, results about the one-dimensional case are rather satisfactory, see \cite{{Hoff 1987}, {Hoff-Smoller}, {Kanel}, {Kazhikhov-Shelukhin}, {Liu-Smoller}} and the references therein. For the multi-dimensional case, Matsumura-Nishida \cite{Matsumura 1980} first showed
the global existence of classical solutions with small initial data, and then Hoff \cite{Hoff 1997} extended the result \cite{Matsumura 1980}
to the discontinuous initial data case. See also \cite{{Ducomet 2010}, {Ducomet 2011}, {Hoff 1992}} for the spherically symmetric case. For the large initial data which may contain vacuum, the global
existence of weak solutions was first proved by Lions \cite{Lions 1998 Compressible models} for the isentropic case, i.e., $p=A\varrho^\gamma$ with $\gamma\geq3/2$ in two dimensions and $\gamma\geq9/5$ in three dimensions. This result was extended to $\gamma>1$ for the spherically symmetric case by Jiang-Zhang \cite{Jiang Song and Zhang Ping} and $\gamma>3/2$ for the general three-dimensional case by Feireisl-Novotn\'{y}-Petzeltov\'{a} \cite{Feireisl 2001}.

In all the works mentioned above, the viscosity coefficients are assumed to be positive constants, which plays an essential role in obtaining fine estimates for the gradient of the velocity field. The global existence becomes more challenging if the viscosity coefficients are effective functions of the density due to the possible occurrence of vacuum.
The one-dimensional case and the multi-dimensional case with spherically symmetric data were studied in \cite{{Amosov}, {Guo-Jiu-Xin}, {Jiang 1998}, {Jiang-Xin-Zhang 2005}, {Liu-Xin-Yang 1998}, {Matsumura-Yanagi}, {Yang 2000}} and the references therein. For the multi-dimensional case with general data, Bresch-Desjardins \cite{B-D 2003} established the global weak solutions, where a new entropy inequality (BD entropy) was obtained to yield more regularity for the density. Then, based on the compactness arguments in Mellet-Vasseur \cite{Mellet-Vasseur}, Li-Xin \cite{Li and Xin} and Vasseur-Yu \cite{Vasseur and Yu} proved the global existence of weak solutions to the compressible Navier-Stokes equations with the viscous Saint-Venant system for the shallow water contained. We also mention a very remarkable result by Vaigant-Kazhikhov \cite{Veigant-Kazhikhov} where the global well-posedness of classical solutions to a potential barotropic compressible model was obtained if the initial density is uniformly away from vacuum and the viscosity coefficients satisfy
\begin{equation*}
\mu=\text{constant} >0, \quad \lambda(\varrho)=\varrho^\beta, \,\beta>3.
\end{equation*}
This result was further developed by Jiu-Wang-Xin \cite{Jiu-Wang-Xin JDE}-\cite{Jiu-Wang-Xin Physica D}.

When the heat-conducting effects are considered and the viscosity coefficients are functions of the temperature,
Feireisl \cite{Feireisl On the motion 2004} proved the global existence of ``variational" solutions to the compressible Navier-Stokes equations
in a bounded domain $\Omega\subset\mathbb{R}^N$, $N=2, 3$, where the viscosity coefficients satisfy
\begin{equation*}
\mu(\vartheta)\geq\underline{\mu}>0, \quad \lambda(\vartheta)+\frac{2}{N}\mu(\vartheta)\geq 0,
\end{equation*}
for some positive constant $\underline{\mu}>0$. The concept of ``variational" solutions was first proposed by Feireisl \cite{Feireisl Dynamics 2004}, where the global existence of ``variational" solutions was proved for the constant viscosity coefficients.

In the present paper, we consider the global existence of weak solutions to the initial-boundary problem \eqref{1.1},  \eqref{1.4}  and \eqref{1.3}, with the viscosity coefficients $\mu(\vartheta)$ and $\lambda(\vartheta)$ satisfying \eqref{1.002}. Note that the shear viscosity coefficient $\mu(\vartheta)$ is degenerate and may vanish in the region of absolutely zero temperature. This assumption is based on the fact that zero viscosity may occur when the temperature is very low, as for some superfluids mentioned in \cite{Ba}.

As the shear viscosity coefficient vanishes, the parabolicity of the momentum equation $\eqref{1.1}_2$ will degenerate. To overcome this difficulty, we add an artificial viscosity term $\eta\Delta\mathbf{u}$ ($\eta>0$) in the momentum equation, which makes it possible to apply the Galerkin method to prove the global solvability of the approximate system. Besides, such an artificial viscosity term  plays an essential role in obtaining the $ L^2(0,T; H_0^1(\Omega))$  estimate for $\mathbf u$ in Sections 2 and 3. However, the artificial term also brings two key problems: first, suitable estimates on $\nabla\mathbf{u}$ independent of the parameter $\eta>0$ are needed; second, possible density oscillation as the artificial viscosity term vanishes
requires extra attention.

To obtain suitable estimates on $\nabla\mathbf{u}$ independent of $\eta>0$, it suffices to get a uniform lower bound for the viscosity coefficient $\mu(\vartheta)$ with respect to $\eta>0$, which can be achieved by proving that the temperature is bounded away from zero by a positive constant. The required positive lower bound for the temperature is obtained in Section 4 by De Giorgi's iteration, a useful method first established by De Giorgi \cite{De} in obtaining H\"{o}lder regularity of solutions to elliptic equation with discontinuous coefficients, and then applied  to the study of the compressible Navier-Stokes equations by Mellet and Vasseur \cite{Vasseur temperature}. It is worthy mentioning that the proof of \cite{Vasseur temperature} relies heavily on the following thermal energy inequality
\begin{equation*}
\begin{split}
&\int_\Omega \varrho\phi(\vartheta)(t,x)dx
-\int_s^t\int_\Omega 2\mu\phi'(\vartheta)|D(\mathbf{u})|^2dxd\tau
-\int_s^t\int_\Omega\lambda\phi'(\vartheta)|{\rm div}\mathbf{u}|^2dxd\tau\\
&+\int_s^t\int_\Omega\kappa\phi''(\vartheta)|\nabla\vartheta|^2dxd\tau
\leq -R\int_s^t\int_\Omega \varrho\vartheta\phi'(\vartheta){\rm div}\mathbf{u} dxd\tau
+ \int_\Omega \varrho\phi(\vartheta)(s,x)dx,
\end{split}
\end{equation*}
for $\phi(\vartheta)=\left[\ln \left(\frac{C}{\vartheta+\varepsilon}\right)\right]_{+}$,
which will be replaced by the modified temperature inequality \eqref{3-11} in our proof.

The way to deal with possible density oscillation is mostly based on the weak continuity property of the effective viscous pressure
\begin{equation*}
P_{eff}=p-(\lambda+2\mu){\rm div}\mathbf{u},
\end{equation*}
which was first introduced by Lions \cite{Lions 1998 Compressible models} for the barotropic case with $\mu$ and $\lambda$ being positive constants, and then generalized by Feireisl \cite{Feireisl On the motion 2004} to the case that $\mu$ and $\lambda$ depend on the temperature with $\mu(\vartheta)\geq \underline{\mu}>0$ for some positive constant $\underline{\mu}$.
The assumption in \cite{Feireisl On the motion 2004} that $\mu(\vartheta)$ has a positive lower bound is essential to ensure the weak continuity property of $P_{eff}$, and this may fail for our degenerate case \eqref{1.002}. However, if we add an artificial viscosity term $\eta\Delta\mathbf{u}$ in the momentum equation, then the term $\mu(\vartheta)+\eta$ can be viewed as a new shear viscosity coefficient with a positive lower bound $\eta$, and therefore the weak continuity property of $P_{eff}$ still holds. In Section 4 ($\eta\to 0$),  we will prove that the temperature is bounded away from zero by De Giorgi's method, which combined with the strictly increasing assumption of the shear viscosity coefficient with respect to the temperature in Theorem \ref{1.2.} implies that there exists a positive constant $\underline{\underline{\mu}}>0$ independent of $\eta$ such that
\begin{equation}\label{1-00}
\mu(\vartheta)\geq\underline{\underline{\mu}}>0.
\end{equation}
Based on \cite{Feireisl On the motion 2004}, the bound \eqref{1-00} implies the strong convergence of the density as the artificial viscosity term $\eta\Delta\mathbf{u}$ vanishes.

Another difficulty in establishing the global existence of weak solutions lies in the temperature concentration, which as in \cite{Feireisl Dynamics 2004} can be tackled by the renormalization of the temperature equation $\eqref{1.1}_3$. Concretely, multiplying $\eqref{1.1}_3$ by $h(\vartheta)$ for some suitable function $h$, we have
\begin{equation*}
\begin{split}
\partial_t(\varrho H(\vartheta))+{\rm div}(\varrho\mathbf{u}H(\vartheta))-\Delta\mathcal{K}_h(\vartheta)
=h(\vartheta)\mathbb{S}:\nabla\mathbf{u}-h'(\vartheta)\kappa(\vartheta)|\nabla\vartheta|^2
-h(\vartheta)\vartheta p_\vartheta {\rm div}\mathbf{u},
\end{split}
\end{equation*}
where
\begin{equation*}
H(\vartheta)=\int_0^\vartheta h(z) dz,\quad\quad \mathcal{K}_h(\vartheta)=\int_0^\vartheta \kappa(z)h(z) dz.
\end{equation*}
The idea of renormalization was first proposed by DiPerna and Lions in \cite{DiPerna and Lions}, where they replaced the continuity equation $\eqref{1.1}_1$ by a renormalized version
\begin{equation*}
\partial_t b(\varrho)+{\rm div}(b(\varrho)\mathbf{u})+(b'(\varrho)\varrho-b(\varrho)){\rm div}\mathbf{u}=0,
\end{equation*}
for suitable functions $b=b(\varrho)$, and then such a method was applied by Lions
\cite{Lions 1996 Incompressible models} and Feireisl \cite{Feireisl Dynamics 2004} to overcome the temperature concentration.

The weak solutions to the initial-boundary value problem \eqref{1.1}, \eqref{1.4} and \eqref{1.3} are defined as follows:
\begin{definition}\label{1.1.}
We call $(\varrho, \mathbf{u}, \vartheta)$ a weak solution to the initial-boundary value problem \eqref{1.1},  \eqref{1.4} and \eqref{1.3} if
\begin{enumerate}[(i)]

\item the density $\varrho\geq0$ satisfies
\begin{equation*}
\varrho\in L^\infty(0,T; L^\gamma(\Omega))\cap C([0,T];L^1(\Omega)),
\end{equation*}
the velocity $\mathbf{u}$ belongs to $L^2(0,T; H^1_0(\Omega))$, and $(\varrho, \mathbf{u})$ solves the continuity equation $\eqref{1.1}_1$ in the sense of distributions, that is, for any $\Phi\in C^\infty_c((0,T)\times\Omega)$
\begin{equation*}
\int_0^T\int_\Omega \varrho \partial_t\Phi dxdt+\int_0^T\int_\Omega \varrho\mathbf{u}\cdot\nabla \Phi dxdt=0;
\end{equation*}

\item the momentum equation $\eqref{1.1}_2$ holds in $\mathcal{D}'((0,T)\times\Omega)$, that means,
\begin{equation*}
\begin{split}
&\int_0^T\int_\Omega \varrho\mathbf{u}\cdot \partial_t\Phi + \varrho\mathbf{u}\otimes\mathbf{u}:\nabla\Phi + p\,{\rm div}\Phi dxdt
=\int_0^T\int_\Omega \mathbb{S}:\nabla \Phi  dxdt,
\end{split}
\end{equation*}
for any $\Phi\in C^\infty_c((0,T)\times\Omega)$. Moreover, $\varrho\mathbf{u}\in C([0,T];L_{weak}^{\frac{2\gamma}{\gamma+1}}(\Omega))$ satisfies the initial condition \eqref{1.4};

\item the temperature $\vartheta\geq 0$ satisfies
\begin{equation*}
\vartheta \in L^2(0,T; H^1(\Omega)),\quad \varrho\vartheta\in L^\infty(0,T; L^1(\Omega)),
\end{equation*}
and $\vartheta(t,\cdot)\rightarrow \vartheta_{0}$ in $\mathcal{D}'(\Omega)$, as $t\rightarrow 0^+$,
that is, for any $\chi\in C_c^\infty(\Omega)$, it holds
\begin{equation*}
\lim_{t\rightarrow 0^+}\int_{\Omega}\vartheta(t,x)\chi(x) dx=\int_{\Omega}\vartheta_0(x)\chi(x)dx.
\end{equation*}
Furthermore, the following temperature inequality holds
\begin{equation*}
\begin{split}
&\int_0^T\int_\Omega \varrho\vartheta\partial_t\varphi dxdt+\int_0^T\int_\Omega \left(\varrho\mathbf{u}\vartheta\cdot\nabla\varphi
+\mathcal{K}(\vartheta)\Delta\varphi\right) dxdt\\
&\leq \int_0^T\int_\Omega \left(\vartheta p_\vartheta {\rm div}\mathbf{u}-\mathbb{S}:\nabla\mathbf{u}\right)\varphi dxdt
-\int_\Omega \varrho_0\vartheta_0\varphi(0)dx,
\end{split}
\end{equation*}
for any $\varphi\in C_c^\infty([0,T]\times\Omega)$ satisfying
\begin{equation*}
\varphi\geq 0,\,\,\varphi(T,\cdot)=0,\,\,\nabla\varphi\cdot \mathbf{n}|_{\partial\Omega}=0,
\end{equation*}
where
\begin{equation*}
\varrho\mathbb{S}=\varrho\left[\mu(\vartheta)(\nabla\mathbf{u}+\nabla^T\mathbf{u})+\lambda(\vartheta){\rm div}\mathbf{u}\mathbb{I}\right],
\end{equation*}
and
\begin{equation*}
\mathcal{K}(\vartheta)=\int_0^\vartheta \kappa(z) dz;
\end{equation*}

\item the energy inequality holds, that is, for a.e. $t\in(0,T)$
\begin{equation*}
E[\varrho, \mathbf{u}, \vartheta](t)\leq E[\varrho, \mathbf{u}, \vartheta](0),
\end{equation*}
where
\begin{equation*}
E[\varrho, \mathbf{u}, \vartheta](t)=\int_\Omega \varrho\left(\frac{1}{2}|\mathbf{u}|^2+P_e(\varrho)+\vartheta\right)(t) dx,
\end{equation*}
and
\begin{equation*}
E[\varrho, \mathbf{u}, \vartheta](0)=\int_\Omega \left(\frac{1}{2}\frac{|\mathbf{m}_0|^2}{\varrho_0}+\varrho_0P_e(\varrho_0)+\varrho_0\vartheta_0\right)dx,
\end{equation*}
with
\begin{equation*}
P_e(\varrho)=\int_1^\varrho\frac{p_e(z)}{z^2}dz.
\end{equation*}

\end{enumerate}
\end{definition}

\begin{remark}\label{1.03.}
The weak solutions in Definition \ref{1.1.} are similar to the ``variational" solutions in \cite{Feireisl Dynamics 2004}.
\end{remark}

\begin{remark}\label{1.4.}
As pointed out in \cite{Feireisl on the Navier-Stokes 2006}, the reason
for introducing the function $\mathcal{K}(\vartheta)=\int_0^\vartheta \kappa(z) dz$
with $\nabla\mathcal{K}(\vartheta)=\kappa(\vartheta)\nabla\vartheta=-\mathbf{q}$
is that we are unable to deduce $\kappa(\vartheta)\nabla\vartheta$ is
locally integrable by a priori estimates. However, we can deduce $\mathcal{K}(\vartheta)\in L^1((0,T)\times\Omega)$ by
constructing proper approximate equations and requiring suitable growth restrictions on $\kappa(\vartheta)$.
\end{remark}

\begin{remark}\label{1.3.}
As will be shown later, the bounds on the velocity that we can obtain from a priori estimates  fail  to ensure the convergence of the term $\mathbb{S}:\nabla\mathbf{u}$ in the sense of distributions. Therefore, as in \cite{Feireisl Dynamics 2004, Feireisl On the motion 2004}, we replaced the temperature equation (\ref{1.1}d) by the following two inequalities in Definition \ref{1.1.}
\begin{equation}\label{1.19}
\partial_t(\varrho\vartheta)+{\rm div}(\varrho\mathbf{u}\vartheta)-\Delta\mathcal{K}(\vartheta)
\geq\mathbb{S}:\nabla\mathbf{u}-\vartheta p_\vartheta(\varrho){\rm div}\mathbf{u},
\end{equation}
and
\begin{equation}\label{1.20}
E[\varrho, \mathbf{u}, \vartheta](t)\leq E[\varrho, \mathbf{u}, \vartheta](0).
\end{equation}
\end{remark}

Before giving our main result, we need to state our assumptions on the pressure.
Throughout this paper, we assume that the pressure $p$ satisfies one of the following two conditions
\begin{itemize}
\item[(1)] $p$ has the form
\begin{equation}\label{1.2900}
p(\varrho, \vartheta)= p_e(\varrho)+\vartheta p_\vartheta(\varrho),
\end{equation}
where $p_e$ and $p_\vartheta$ satisfy
\begin{equation}\label{1.30}
\begin{cases}
p_e\in C[0,\infty)\cap C^1(0,\infty),\\
p_e(0)=0,\,p_e'(\varrho)\geq a_1\varrho^{\gamma-1}-b \,\,&{\rm for\,\,all}\,\,\varrho>0,\\
p_e(\varrho)\leq a_2\varrho^{\gamma}+b \,\,&{\rm for\,\,all}\,\,\varrho\geq0,
\end{cases}
\end{equation}

\begin{equation}\label{1.31}
\begin{cases}
p_\vartheta\in C[0,\infty)\cap C^1(0,\infty),\\
p_\vartheta(0)=0,\,p_\vartheta'(\varrho)\geq 0\,\,&{\rm for\,\,all}\,\,\varrho>0,\\
p_\vartheta(\varrho)\leq c(1+\varrho^{\gamma/3})\,\,&{\rm for\,\,all}\,\,\varrho\geq0,
\end{cases}
\end{equation}
where $\gamma>\frac{3}{2}$ and $a_1$, $a_2$, $b$, $c$ are positive constants;
\item[(2)]$p$ has the form
\begin{equation}\label{1.310}
p(\varrho, \vartheta)= p_e(\varrho)+R\varrho\vartheta
\end{equation}
where $p_e(\varrho)$ satisfies \eqref{1.30} with $a_1,a_2,b$ being positive constants, $\gamma>3$ and $R$ is a positive constant.
\end{itemize}
It can be easily checked that  the above two conditions are independent.

Now we are ready to state our main result.
\begin{theorem}\label{1.2.}
Let $\Omega\subset \mathbb{R}^3$ be a smooth bounded domain. Assume that

\begin{enumerate}[(i)]
\item the pressure $p$ satisfies \eqref{1.2900} or \eqref{1.310};

\item the heat conductivity coefficient $\kappa(\vartheta)\in C^1[0,\infty)$ and satisfies
\begin{equation}\label{1.32}
\underline{\kappa}(1+\vartheta^2)\leq\kappa(\vartheta)\leq\overline{\kappa}(1+\vartheta^2),
\end{equation}
for positive constants $\underline{\kappa}$ and $\overline{\kappa}$;

\item the viscosity coefficients $\mu$ and $\lambda$ are globally Lipschitz continuous functions on $[0,\infty)$ and $\mu$ is strictly increasing, moreover, $\mu$ and $\lambda$ satisfy
\begin{equation}\label{1.33}
2\mu(\vartheta)+3\lambda(\vartheta)\geq\nu(\vartheta)>0\quad{\rm for \,\,all \,\,}\vartheta,
\end{equation}
\begin{equation}\label{1.330}
\nu(\vartheta)\geq C\vartheta\quad\quad\quad\quad\quad\quad{\rm for \,\,small \,\,}\vartheta,
\end{equation}
for some function $\nu:[0,+\infty)\to\mathbb R$ and positive constant $C$;

\item the initial data satisfy
\begin{equation}\label{1.34}
\begin{cases}
\varrho_0\in L^\gamma(\Omega), \,\,\varrho_0\geq \underline{\varrho}>0  \,\,&{\rm on}\,\,\Omega,\\
\vartheta_0\in L^1(\Omega),\,\,\vartheta_0\geq\underline{\vartheta}>0 \,\,& {\rm on}\,\,\Omega,\\
\
\frac{|\mathbf{m}_0|^2}{\varrho_0}\in L^1(\Omega)
\end{cases}
\end{equation}
\end{enumerate}
for positive constants $\underline{\varrho}$ and $\underline{\vartheta}$.
Then for any given $T>0$, the initial-boundary value problem \eqref{1.1}, \eqref{1.4} and \eqref{1.3} admits a global weak solution $(\varrho,\mathbf{u},\vartheta)$ in the sense of Definition \ref{1.1.}.
\end{theorem}

\begin{remark}\label{1}
Zero viscosity only occurs when the temperature is very low in superfluids as in \cite{Ba};
Otherwise, the viscosity of all fluids is positive by the second law of thermodynamics.
Thus our restrictions \eqref{1.33} and \eqref{1.330} in Theorem \ref{1.2.} are physical and reasonable.
\end{remark}

\begin{remark}\label{2}
The strictly increasing assumption of the shear viscosity coefficient $\mu$ in Theorem \ref{1.2.} is reasonable. In fact, gases are generally considered as compressible fluid,  and in most cases the shear viscosity of gases increases as the temperature increases; See \cite{Reid} for example.
\end{remark}

We give some comments on Theorem \ref{1.2.} by comparing it with two closely related results of the full compressible Navier-Stokes equations in \cite{Feireisl On the motion 2004} and \cite{Vasseur temperature}.
In \cite{Feireisl On the motion 2004}, Feireisl proved the global existence of the so-called ``variational" solutions where the viscosity coefficients are functions of the temperature satisfying
\begin{equation*}
\mu(\vartheta)\geq\underline{\mu}>0,\quad \kappa(\vartheta)+\frac{2}{3}\mu(\vartheta)\geq0,
\end{equation*}
and the pressure is given by \eqref{1.2900} with $p_e$ and $p_\vartheta$ satisfying \eqref{1.30} and \eqref{1.31}.
By contrast, Theorem \ref{1.2.}
does not require that the shear viscosity has a positive lower bound and allows another choice for the pressure (i.e., \eqref{1.30} and \eqref{1.310}),
and thus can be regarded as an improvement of Feireisl's result. The other related result is obtained by
Mellet and Vasseur \cite{Vasseur temperature}, where they used De Giorgi's method to show that the temperature is uniformly positive in any finite time interval provided that the initial temperature has a lower bound away from zero.
The assumptions on $\varrho, \mathbf{u}$ and $p$ in \cite{Vasseur temperature} are
\begin{equation}\label{1-1}
\varrho\in L^\infty(0,T;L^p(\Omega)),\quad{\rm for\,\,some\,\,}p>3,
\end{equation}
\begin{equation}\label{1-2}
\mathbf{u}\in L^2(0,T;H_0^1(\Omega)),
\end{equation}
\begin{equation}\label{1-3}
p=p(\varrho,\vartheta)=\tilde p_e(\varrho)+R\varrho\vartheta,
\end{equation}
where $\tilde p_e$ can be any function of $\varrho$.
These assumptions on solutions are independent of ours when we impose the first kind of conditions on $p$, that is, \eqref{1.2900},\eqref{1.30} and \eqref{1.31}.
When $p$ satisfies the second kind of conditions, that is, \eqref{1.30} and \eqref{1.310}, the assumptions \eqref{1-1}-\eqref{1-3} are weaker than ours, but it seems that they are not enough to get the global existence.

Our paper is organized as follows.
In Section 2, we construct a suitable approximate system \eqref{2-1}-\eqref{2-302} with three parameters $\varepsilon,\eta$ and $\delta$
and obtain its global solvability by means of a modified Galerkin method.
In Section 3, we let $\varepsilon\rightarrow 0$ for the approximate solutions constructed in Section 2.
In Section 4, we apply De Giorgi's iteration to obtain a positive lower bound for the temperature not depending on the parameter $\eta$ and then pass to the limit $\eta\rightarrow 0$.
In Section 5, we let $\delta\rightarrow 0$ to finish the proof of Theorem \ref{1.2.}.

\section{Construction of approximate solutions}
\setcounter{equation}{0}
First, we construct the following approximate system:
\begin{enumerate}[(i)]
\item continuity equation with vanishing viscosity:
\begin{equation}\label{2-1}
\partial_t\varrho+{\rm div}(\varrho\mathbf{u})=\varepsilon\Delta\varrho, \quad\varepsilon>0,
\end{equation}
\begin{equation}\label{2-101}
\nabla\varrho\cdot\mathbf{n}=0\quad{\rm on}\,\,\partial\Omega,
\end{equation}
\begin{equation}\label{2-102}
\varrho(0,\cdot)=\varrho_{0,\delta}\quad{\rm in}\,\,\Omega;
\end{equation}

\item momentum equation with artificial pressure and artificial viscosity:
\begin{equation}\label{2-2}
\partial_t(\varrho\mathbf{u})+{\rm div}(\varrho\mathbf{u}\otimes\mathbf{u})
+\nabla(p(\varrho,\vartheta)+\delta\varrho^\beta)+\varepsilon\nabla\mathbf{u}\nabla\varrho
=\eta\Delta\mathbf{u}+{\rm div}\mathbb{S},\quad\eta,\delta>0,
\end{equation}
\begin{equation}\label{2-201}
\mathbf{u}=0\quad{\rm on}\,\,\partial\Omega,
\end{equation}
\begin{equation}\label{2-202}
(\varrho\mathbf{u})(0,\cdot)=\mathbf{m}_{0,\delta}\quad{\rm in}\,\,\Omega;
\end{equation}

\item regularized temperature equation:
\begin{equation}\label{2-3}
\partial_t((\delta+\varrho)\vartheta)+{\rm div}(\varrho\mathbf{u}\vartheta)-\Delta \mathcal{K}(\vartheta)+\delta \vartheta^3
=(1-\delta)\mathbb{S}:\nabla\mathbf{u}-\vartheta p_\vartheta(\varrho){\rm div}\mathbf{u},
\end{equation}
\begin{equation}\label{2-301}
\nabla\vartheta\cdot\mathbf{n}=0\quad{\rm on}\,\,\partial\Omega,
\end{equation}
\begin{equation}\label{2-302}
\vartheta(0,\cdot)=\vartheta_{0,\delta}\quad{\rm in}\,\,\Omega.
\end{equation}
\end{enumerate}

Note that the construction of the above approximate system is motivated by but different from
\cite{Feireisl 2001}-\cite{Feireisl Singular}.
Moreover, the modified initial data are required to satisfy the following conditions:
\begin{equation}\label{2-4}
\begin{cases}
\varrho_{0,\delta}\in C^{2+\nu}(\bar{\Omega}),\nu>0, \quad \nabla\varrho_{0,\delta}\cdot\mathbf{n}|_{\partial\Omega}=0,\quad 0<\delta\leq\varrho_{0,\delta}\leq \delta^{-1/{2\beta}};\\
\varrho_{0,\delta}\rightarrow\varrho_{0}\,\,{\rm in}\,\,L^\gamma(\Omega),\quad
|\{x\in\Omega\,|\,\varrho_{0,\delta}<\varrho_0\}|\rightarrow 0,\,\,{\rm as}\,\,\delta\rightarrow 0;\\
\vartheta_{0,\delta}\in C^{2+\nu}(\bar{\Omega}),\quad \nabla\vartheta_{0,\delta}\cdot\mathbf{n}|_{\partial\Omega}=0,\quad
0<\underline{\vartheta}\leq\vartheta_{0,\delta}\leq\overline{\vartheta};\\
\vartheta_{0,\delta}\rightarrow\vartheta_{0}\,\,{\rm in}\,\,L^1(\Omega),\,\, {\rm as}\,\,\delta\rightarrow 0;\\
\mathbf{m}_{0,\delta}=
\begin{cases}
\mathbf{m}_0,\,\,& {\rm if}\,\,\varrho_{0,\delta}\geq\varrho_0,\\
0,\,\,& {\rm if}\,\,\varrho_{0,\delta}<\varrho_0,
\end{cases}
\end{cases}
\end{equation}
where the positive constant $\underline{\vartheta}$ is independent of $\delta>0$.
In particular, the regularized initial value of the total energy
\begin{equation}\label{initial}
E(0)=E_{\delta}(0)=\int_\Omega \left(\frac{1}{2}\frac{|\mathbf{m}_{0,\delta}|^2}{\varrho_{0,\delta}}
+\varrho_{0,\delta} P_e(\varrho_{0,\delta})
+\frac{\delta}{\beta-1}\varrho_{0,\delta}^\beta
+\varrho_{0,\delta}\vartheta_{0,\delta} \right)dx
\end{equation}
is bounded by a constant independent of $\delta>0$.

\begin{remark}\label{2.1.}
Roughly speaking, the extra term $\varepsilon\Delta\varrho$ is introduced in \eqref{2-1} to convert the hyperbolic equation $\eqref{1.1}_1$ into a parabolic one from which one can obtain better regularity property of the density $\varrho$. The quantity $\varepsilon\nabla\mathbf{u}\nabla\varrho$ is added to \eqref{2-2} to eliminate the term related to $\varepsilon\Delta\varrho$ in the energy inequality.
The new quantity $\eta\Delta\mathbf{u}$ represents the artificial viscosity which ensures the parabolic property of the momentum equation \eqref{2-2} and the term $\delta\nabla\varrho^\beta$ represents an artificial pressure which will play an essential role in obtaining estimates for the density $\varrho$.
The term $\delta\vartheta^3$ is introduced to improve the integrability of the temperature. The other terms related to the parameter $\delta>0$ are introduced to avoid technicalities in the temperature estimates.
\end{remark}

Similarly to Chapter 7 in \cite{Feireisl Dynamics 2004}, for fixed positive parameters $\varepsilon$, $\eta$ and $\delta$, the global solvability of the approximate system \eqref{2-1}-\eqref{2-302} can be obtained by Galerkin method and the result is as follows.
\begin{proposition}\label{2.1.}
For fixed positive parameters $\varepsilon$, $\eta$ and $\delta$, under the hypotheses of Theorem \ref{1.2.} and the assumptions imposed on the initial data \eqref{2-4}, if the exponent
\begin{equation*}
\beta>\max\{4,\gamma\},
\end{equation*}
then the approximate system \eqref{2-1}-\eqref{2-302} admits a global weak solution $(\varrho, \mathbf{u}, \vartheta)$ satisfying the following properties:
\begin{enumerate}[(i)]
\item the density $\varrho\geq 0$ satisfies
\begin{equation*}
\partial_t\varrho,\,\,\Delta\varrho\in L^p((0,T)\times\Omega)\quad{\rm for\,\,a\,\,certain\,\,}p>1,
\end{equation*}
the velocity $\mathbf{u}$ belongs to the space $L^2(0,T; H^1_0(\Omega))$, $(\varrho, \mathbf{u})$ solves the modified continuity equation \eqref{2-1} a.a. on $(0,T)\times\Omega$, and the boundary condition \eqref{2-101} together with the initial condition \eqref{2-102} are satisfied in the sense of traces. Moreover,
\begin{equation*}
\delta \int_0^T\int_\Omega \varrho^{\beta+1} dxdt \leq C(\varepsilon, \delta),
\end{equation*}
\begin{equation*}
\varepsilon \int_0^T\int_\Omega |\nabla\varrho|^2 dxdt \leq C,\quad{\rm with}\,\,C\,\,{\rm independent\,\,of}\,\, \varepsilon>0;
\end{equation*}

\item $(\varrho, \mathbf{u}, \vartheta)$ solves the modified momentum equation \eqref{2-2} in $\mathcal{D}'((0,T)\times\Omega)$. Moreover,
    \begin{equation*}
    \varrho\mathbf{u}\in C([0,T];L_{weak}^{\frac{2\gamma}{\gamma+1}}(\Omega))
    \end{equation*}
satisfies the initial condition \eqref{2-202};

\item the temperature $\vartheta\geq0$ satisfies
\begin{equation*}
\vartheta\in L^3((0,T)\times\Omega),\quad \vartheta\in L^2(0,T;H^1(\Omega)),
\end{equation*}
and the renormalized temperature inequality holds in $\mathcal{D}'((0,T)\times\Omega)$, that is,
\begin{equation}\label{2.8}
\begin{split}
&\int_0^T\int_\Omega (\delta+\varrho)H(\vartheta)\partial_t\varphi dxdt\\
&\quad+\int_0^T\int_\Omega \left(\varrho H(\vartheta)\mathbf{u}\cdot\nabla\varphi
+\mathcal{K}_h(\vartheta)\Delta\varphi-\delta\vartheta^3 h(\vartheta)\varphi\right) dxdt\\
&\leq\int_0^T\int_\Omega \left((\delta-1)\mathbb{S}:\nabla\mathbf{u}h(\vartheta)+h'(\vartheta)\kappa(\vartheta)|\nabla\vartheta|^2\right)\varphi dxdt\\
&\quad+\int_0^T\int_\Omega h(\vartheta)\vartheta p_\vartheta(\varrho){\rm div}\mathbf{u}\varphi dxdt
+\varepsilon \int_0^T\int_\Omega \nabla\varrho\cdot\nabla((H(\vartheta)-\vartheta h(\vartheta))\varphi) dxdt\\
&\quad-\int_\Omega(\delta+\varrho_{0,\delta})H(\vartheta_{0,\delta})\varphi(0)dx
\end{split}
\end{equation}
for any $\varphi\in C_c^\infty([0,T]\times\Omega)$ satisfying
\begin{equation}\label{2.9}
\varphi\geq 0,\,\,\varphi(T,\cdot)=0,\,\,\nabla\varphi\cdot \mathbf{n}|_{\partial\Omega}=0,
\end{equation}
where $H(\vartheta)=\int_0^\vartheta h(z) dz$ and $\mathcal{K}_h(\vartheta)=\int_0^\vartheta \kappa(z)h(z) dz$,
with the non-increasing $h\in C^2([0,\infty))$ satisfying
\begin{equation}\label{2.10}
0<h(0)<\infty,\,\,\lim_{z\rightarrow\infty}h(z)=0,\\
\end{equation}
and
\begin{equation}\label{2.01}
h''(z)h(z)\geq 2(h'(z))^2\,\,for \,\,all\,\,z\geq0;
\end{equation}

\item the energy inequality
\begin{equation}\label{2.11}
\begin{split}
&\int_0^T\int_\Omega(-\partial_t\psi)\left(\frac{1}{2}\varrho|\mathbf{u}|^2+\varrho P_m(\varrho)+\frac{\delta}{\beta-1}\varrho^\beta
+(\delta+\varrho)\vartheta \right)dxdt\\
&\quad+\int_0^T\int_\Omega \psi\left(\delta\mathbb{S}:\nabla\mathbf{u}+\eta|\nabla\mathbf{u}|^2+\delta \vartheta^3\right) dxdt\\
&\leq\int_\Omega \left(\frac{1}{2}\frac{|\mathbf{m}_{0,\delta}|^2}{\varrho_{0,\delta}}
+\varrho_{0,\delta} P_m(\varrho_{0,\delta})
+\frac{\delta}{\beta-1}\varrho_{0,\delta}^\beta
+(\delta+\varrho_{0,\delta})\vartheta_{0,\delta} \right)dx\\
&\quad-\int_0^T\int_\Omega p_b(\varrho){\rm div}\mathbf{u}\psi dxdt
\end{split}
\end{equation}
holds for any $\psi\in C^\infty([0,T])$ satisfying
\begin{equation}\label{2.1100}
\psi(0)=1,\quad \psi(T)=0,\quad \partial_t\psi\leq 0,
\end{equation}
where the elastic pressure component $p_e$ has been written in the form
\begin{equation*}
p_e(\varrho)=p_m(\varrho)-p_b(\varrho),
\end{equation*}
with $p_m$ non-decreasing, $P_m(\varrho)=\int_1^\varrho\frac{p_m(z)}{z^2}dz$, and
\begin{equation*}
p_b\in C^2([0,\infty)),\quad p_b\geq0,\quad
\end{equation*}
compactly supported in $[0,\infty)$.
\end{enumerate}
\end{proposition}

\begin{remark}\label{1.6.}
As proved in \cite{Feireisl Dynamics 2004} for the constant viscosity coefficients case, the hypothesis \eqref{2.01} is imposed to ensure the convex and weakly lower semi-continuous property of the function
\begin{equation}\label{1-10}
(\vartheta, \nabla\mathbf{u})\mapsto h(\vartheta)\mathbb{S}:\nabla\mathbf{u},
\end{equation}
which is still valid for the temperature-depending viscosity coefficients case {\rm (cf. \cite{Xianpeng Hu and Dehua Wang})}.
\end{remark}

\section{Passing to the limit for $\varepsilon\to0$}\
\setcounter{equation}{0}
In this section, our main task is to pass the limit $\varepsilon\to 0$ in the approximate system \eqref{2-1}-\eqref{2-302}. For clarity, we denote by $(\varrho_\varepsilon, \mathbf{u}_\varepsilon, \vartheta_\varepsilon)$ the weak solutions constructed in Proposition \ref{2.1.}. Following Section 4 in \cite{Feireisl On the motion 2004}, the main difficulty lies in obtaining the strong convergence of the density
\begin{equation}\label{3-0000}
\varrho_\varepsilon\to \varrho\quad{\rm in}\,\,L^1((0,T)\times\Omega),
\end{equation}
which can be achieved by the weak continuity property of the quantity
\begin{equation*}
P_{eff}=p-(\lambda+2\mu){\rm div}\mathbf{u}
\end{equation*}
called usually the effective viscous pressure. This method was first proposed and proved by Lions \cite{Lions 1998 Compressible models} for the barotropic case with constant viscous coefficients. More specifically, if
\begin{equation*}
\begin{cases}
\varrho_\varepsilon \rightharpoonup \varrho&\quad{\rm weakly\,\, in\,\,}L^1((0,T)\times\Omega),\\
p(\varrho_\varepsilon, \vartheta_\varepsilon)
\rightharpoonup\overline{p}&\quad{\rm weakly\,\, in\,\,}L^1((0,T)\times\Omega),\\
\mathbf{u}_\varepsilon\rightharpoonup\mathbf{u}&
\quad{\rm weakly\,\,in\,\,}L^2(0,T;H_0^1(\Omega)),\\
\end{cases}
\end{equation*}
then under certain hypotheses it holds
\begin{equation*}
\begin{split}
\left(p(\varrho_\varepsilon, \vartheta_\varepsilon)
-(\lambda+2\mu){\rm div}\mathbf{u}_\varepsilon\right)\varrho_\varepsilon
\rightharpoonup \left(\overline{p}
-(\lambda+2\mu){\rm div}\mathbf{u}\right)\varrho
\end{split}
\end{equation*}
weakly in $L^1((0,T)\times\Omega)$.
Then, Feireisl \cite{Feireisl On the motion 2004} extended the above result to the case where $\mu$ and $\lambda$ are functions of the temperature and satisfy
\begin{equation*}
\mu(\vartheta)\geq\underline{\mu}>0,\quad\lambda(\vartheta)+\frac{2}{3}\mu(\vartheta)\geq 0,
\end{equation*}
for some positive constant $\underline{\mu}>0$.

For fixed $\eta, \delta>0$, the effective viscous pressure corresponding to the approximate system \eqref{2-1}-\eqref{2-302} are as follows
\begin{equation*}
p+\delta\varrho^\beta
-(\lambda(\vartheta)+2\mu(\vartheta)+\eta){\rm div}\mathbf{u}.
\end{equation*}
Based on Section 4 in \cite{Feireisl On the motion 2004} and under assumptions of Theorem \ref{1.2.}, we can show that
\begin{equation}\label{3-007}
\begin{split}
&\lim_{\varepsilon\to 0}\int_0^T\int_\Omega \varphi \left(p(\varrho_\varepsilon, \vartheta_\varepsilon)+\delta\varrho_\varepsilon^\beta
-(\lambda(\vartheta_\varepsilon)+2\mu(\vartheta_\varepsilon)+\eta){\rm div}\mathbf{u}_\varepsilon)\right)\varrho_\varepsilon dxdt\\
&=\int_0^T\int_\Omega \varphi  \left(\overline{p}-(\lambda(\vartheta)+2\mu(\vartheta)+\eta){\rm div}\mathbf{u})\right)\varrho dxdt,
\end{split}
\end{equation}
for any $\varphi\in C_c^\infty((0,T)\times\Omega)$ provided that
\begin{equation*}
\begin{cases}
p(\varrho_\varepsilon, \vartheta_\varepsilon)+\delta\varrho_\varepsilon^\beta\to \overline{p}&\quad {\rm weakly\,\,in\,\,}L^{(\beta+1)/\beta}((0,T)\times\Omega),\\
\varrho_\varepsilon\to \varrho&\quad {\rm in\,\,}C([0,T];L^\beta_{weak}(\Omega)),\\
\mathbf{u}_\varepsilon\rightharpoonup\mathbf{u}&
\quad{\rm weakly\,\,in\,\,}L^2(0,T;H_0^1(\Omega)),\\
\vartheta_\varepsilon\to \vartheta&\quad {\rm in\,\,}L^2((0,T)\times\Omega).\\
\end{cases}
\end{equation*}

With the strong convergence of the density $\varrho_\varepsilon$ \eqref{3-0000}, we conclude our result as follows.

\begin{proposition}\label{3.1.}
For fixed positive parameters $\eta$ and $\delta$, under the hypotheses of Theorem \ref{1.2.},
the initial-boundary value problem \eqref{1.1}, \eqref{1.4} and \eqref{1.3} with parameters $\eta$ and $\delta$ admits an
approximate solution $(\varrho, \mathbf{u}, \vartheta)$, which is also the limit of the weak solution constructed in Proposition \ref{2.1.} when $\varepsilon\to 0$, satisfying
\begin{enumerate}[(i)]
\item the density $\varrho\geq 0$ satisfies
\begin{equation*}
\varrho\in C([0,T];L^\beta_{weak}(\Omega))\cap L^{\beta+1}((0,T)\times\Omega)
\end{equation*}
and the initial condition \eqref{2-102}. The velocity $\mathbf{u}$ belongs to the space
$L^2(0,T; H^1_0(\Omega))$, and $(\varrho, \mathbf{u})$ solves the continuity equation $\eqref{1.1}_1$ in the sense of distributions;

\item $(\varrho, \mathbf{u}, \vartheta)$ solves a modified momentum equation
\begin{equation}\label{3-10}
\partial_t(\varrho\mathbf{u})+{\rm div}(\varrho\mathbf{u}\otimes\mathbf{u})+\nabla(p(\varrho,\vartheta)+\delta\varrho^\beta)
={\rm div}\mathbb{S}+\eta\Delta\mathbf{u}
\end{equation}
in $\mathcal{D}'((0,T)\times\Omega)$, where the viscous stress tensor $\mathbb{S}$ is given by
\begin{equation*}
\mathbb{S}=\mu(\vartheta)(\nabla\mathbf{u}+\nabla^T\mathbf{u})+\lambda(\vartheta){\rm div}\mathbf{u}\mathbb{I}.
\end{equation*}
Moreover, $\varrho\mathbf{u}\in C([0,T]; L^{\frac{2\gamma}{\gamma+1}}_{weak}(\Omega))$ satisfies the initial condition \eqref{2-202};

\item the temperature $\vartheta\geq 0$ satisfies
\begin{equation*}
\vartheta \in L^3((0,T)\times\Omega),\quad \vartheta^{\frac{3-\omega}{2}}\in L^2(0,T;H^1(\Omega)),\,\,\omega\in(0,1),
\end{equation*}
and the initial condition \eqref{2-302} is satisfied in the sense of distributions.
Furthermore, the renormalized temperature inequality holds in $\mathcal{D}'((0,T)\times\Omega)$, that is,
\begin{equation}\label{3-11}
\begin{split}
&\int_0^T\int_\Omega (\delta+\varrho)H(\vartheta)\partial_t\varphi dxdt\\
&\quad+\int_0^T\int_\Omega \left(\varrho H(\vartheta)\mathbf{u}\cdot\nabla\varphi
+\mathcal{K}_h(\vartheta)\Delta\varphi-\delta\vartheta^3 h(\vartheta)\varphi\right) dxdt\\
&\leq\int_0^T\int_\Omega \left((\delta-1)\mathbb{S}:\nabla\mathbf{u}h(\vartheta)+h'(\vartheta)\kappa(\vartheta)
|\nabla\vartheta|^2\right)\varphi dxdt\\
&\quad+\int_0^T\int_\Omega h(\vartheta)\vartheta p_\vartheta(\varrho){\rm div}\mathbf{u}\varphi dxdt -\int_\Omega(\delta+\varrho_{0,\delta})H(\vartheta_{0,\delta})\varphi(0)dx
\end{split}
\end{equation}
for any $\varphi\in C_c^\infty([0,T]\times\Omega)$ satisfying
\begin{equation*}
\varphi\geq 0,\,\,\varphi(T,\cdot)=0,\,\,\nabla\varphi\cdot \mathbf{n}|_{\partial\Omega}=0,
\end{equation*}
where $H(\vartheta)=\int_0^\vartheta h(z) dz$ and $\mathcal{K}_h(\vartheta)=\int_0^\vartheta \kappa(z)h(z) dz$,
with the non-increasing $h\in C^2([0,\infty))$ satisfying
\begin{equation*}
0<h(0)<\infty,\,\,\lim_{z\rightarrow\infty}h(z)=0,\\
\end{equation*}
and
\begin{equation*}
h''(z)h(z)\geq 2(h'(z))^2\,\,for \,\,all\,\,z\geq0;
\end{equation*}

\item the energy inequality
\begin{equation}\label{3-12}
\begin{split}
&\int_0^T\int_\Omega(-\partial_t\psi)
\left(\frac{1}{2}\varrho|\mathbf{u}|^2+\varrho P_e(\varrho)+\frac{\delta}{\beta-1}\varrho^\beta
+(\delta+\varrho)\vartheta \right)dxdt\\
&\quad+\int_0^T\int_\Omega \psi\left(\delta\mathbb{S}:\nabla\mathbf{u}+\eta|\nabla\mathbf{u}|^2+\delta \vartheta^3\right) dxdt\\
&\leq\int_\Omega \left(\frac{1}{2}\frac{|\mathbf{m}_{0,\delta}|^2}{\varrho_{0,\delta}}
+\varrho_{0,\delta} P_e(\varrho_{0,\delta})
+\frac{\delta}{\beta-1}\varrho_{0,\delta}^\beta
+(\delta+\varrho_{0,\delta})\vartheta_{0,\delta} \right)dx
\end{split}
\end{equation}
holds for any $\psi\in C^\infty([0,T])$ satisfying
\begin{equation*}
\psi(0)=1,\quad \psi(T)=0,\quad \partial_t\psi\leq 0.
\end{equation*}
\end{enumerate}
\end{proposition}

\section{Passing to the limit for $\eta\to0$}\
\setcounter{equation}{0}
In this section, our goal is to eliminate the artificial viscosity $\eta\Delta\mathbf{u}$ in the modified momentum equation \eqref{3-10}. To this end, we denote by $(\varrho_\eta, \mathbf{u}_\eta, \vartheta_\eta)$ the weak solutions constructed in Proposition \ref{3.1.}. Note that, for any fixed $\eta>0$, $\sqrt{\eta}\nabla\mathbf{u}$ is bounded in $L^2((0,T)\times\Omega)$, which plays an essential role in obtaining the compactness of weak solutions as $n\to \infty$ and $\varepsilon\to 0$. However, this estimate is not uniform on $\eta>0$. Therefore, to obtain the uniform bound $\nabla\mathbf{u}_{\eta} \in L^2((0,T)\times\Omega)$ with respect to $\eta$, we need to
show that the temperature-depending viscosity coefficient $\mu(\vartheta_\eta)$ is bounded below from zero, which can be achieved by obtaining a uniform positive lower bound for the temperature $\vartheta_\eta$.

\subsection{{\bf A positive bound from below for the temperature}}\

We first give our result about the positive bound for the temperature $\vartheta_\eta$.
\begin{proposition}\label{6.1.}
Let $(\varrho_\eta, \mathbf{u}_\eta, \vartheta_\eta)$ be a weak solution constructed in Proposition \ref{3.1.}. Assume that the viscosity coefficients satisfy the assumptions in Theorem \ref{1.2.}, and the initial temperature satisfies the assumptions in \eqref{2-4}, that is,
\begin{equation}\label{5-0}
\vartheta_\eta(0)=\vartheta_{0,\delta}\geq\underline{\vartheta}>0.
\end{equation}
Then there exists a constant $\widetilde{\vartheta}>0$ such that
\begin{equation}\label{5-1}
\vartheta_{\eta}(t,x)\geq\widetilde{\vartheta}>0
\end{equation}
for all $t\in[0,T]$ and almost all $x\in\Omega$.
\end{proposition}

\begin{remark}\label{6-1}
It should be pointed out that the constant $\widetilde{\vartheta}$ does not depend on the positive parameters $\eta$ and $\delta$, which is essential to take the limit $\eta\to 0$ and $\delta\to 0$.
\end{remark}

Before proving Proposition \ref{6.1.}, we introduce two important lemmas.
\begin{lemma}\label{6.2.}{\rm (\cite{Vasseur  Lp estimates})}
Let $U_k$ be a sequence satisfying
\begin{enumerate}[(i)]
\item $0\leq U_0\leq C$;

\item for some constants $A\geq1$, $1<\beta_1<\beta_2$ and $C>0$,
\begin{equation}\label{6.4}
0\leq U_k\leq C\frac{A^k}{K}(U_{k-1}^{\beta_1}+U_{k-1}^{\beta_2}).
\end{equation}
\end{enumerate}
Then there exists some $K_0$ such that for every $K>K_0$, the sequence $U_k$ converges to $0$ when $k$ goes to infinity.
\end{lemma}

\begin{lemma}\label{6.20.}{\rm (\cite{Feireisl Dynamics 2004})}
Let $\varrho$ be a non-negative function such that
\begin{equation}\label{4.17}
0<M_1\leq \int_\Omega \varrho dx, \int_\Omega \varrho^\gamma dx  \leq M_2,  \,\,{\rm with}\,\,\gamma>\frac{6}{5}.
\end{equation}
Then there exists a positive constant $C$ depending only on $M_1$, $M_2$ such that
\begin{equation}\label{4.18}
\|\mathbf{v}\|_{H^{1}(\Omega)}\leq C\left(\|\nabla\mathbf{v}\|_{L^2(\Omega)}+\int_\Omega \varrho|\mathbf{v}| dx\right).
\end{equation}
\end{lemma}

Our proof is in the spirit of the work of Mellet and Vasseur \cite{Vasseur temperature}, where
they obtained a bound from below for the temperature in the compressible Navier-Stokes equations by De Giorgi's method. Since the case that the pressure $p$ satisfies the second kind of conditions, that is, \eqref{1.30} and \eqref{1.310} has been included in \cite{Vasseur temperature}, we only give the proof of the case that $p$ satisfies \eqref{1.2900},\eqref{1.30} and \eqref{1.31}.

\begin{proof}[Proof of Proposition \ref{6.1.}]
For clarity, we divide the proof into four steps.

\underline{\bf{Step 1.}}
According to the renormalized temperature inequality \eqref{3-11}, $(\varrho_\eta, \mathbf{u}_\eta, \vartheta_\eta)$ satisfies
\begin{equation}\label{6.700}
\begin{split}
&\partial_t((\delta+\varrho_\eta) H(\vartheta_\eta))+{\rm div}(\varrho_\eta\mathbf{u}_\eta H(\vartheta_\eta))-\Delta\mathcal{K}_h(\vartheta_\eta)
-h'(\vartheta_\eta)\kappa(\vartheta_\eta)|\nabla\vartheta_\eta|^2\\
&\leq\delta\vartheta_\eta^3h(\vartheta_\eta)
-(1-\delta)\mathbb{S}_\eta:\nabla\mathbf{u}_\eta h(\vartheta_\eta)
+\vartheta_\eta p_\vartheta(\varrho_\eta){\rm div}\mathbf{u}_\eta h(\vartheta_\eta),
\end{split}
\end{equation}
in $\mathcal{D}'((0,T)\times\Omega)$, where $H(\vartheta)=-\int_0^\vartheta h(z)dz$ and $\mathcal{K}_h(\vartheta)=-\int_0^\vartheta \kappa(z)h(z) dz$, with the non-increasing $h\in C^2([0,\infty))$ satisfying
\begin{equation*}
0<h(0)<\infty,\,\,\lim_{z\rightarrow\infty}h(z)=0,\\
\end{equation*}
and
\begin{equation*}
h''(z)h(z)\geq 2(h'(z))^2\,\,for \,\,all\,\,z\geq0.
\end{equation*}
In particular, the function $h(z)=\frac{1}{z+\omega}1_{\{z+\omega\leq C\}}$ with $\omega>0$ satisfies all the conditions, thus the inequality \eqref{6.700} holds with
\begin{equation}\label{6.7}
H(\vartheta_\eta)=
\begin{cases}
-\ln(\vartheta_\eta+\omega)+\ln\omega,\,\,&{\rm if} \,\,\vartheta_\eta+\omega \leq C,\\
-\ln C +\ln \omega,\,\,&{\rm if} \,\,\vartheta_\eta+\omega > C.\\
\end{cases}
\end{equation}
Taking
\begin{equation}\label{}
\phi(\vartheta_\eta)=H(\vartheta_\eta)+\ln C-\ln \omega=\left[\ln\left(\frac{C}{\vartheta_\eta+\omega}\right)\right]_{+},
\end{equation}
we deduce for any $0\leq s\leq t\leq T$
\begin{equation}\label{6.8}
\begin{split}
&\int_\Omega \left((\delta+\varrho_\eta)\phi(\vartheta_\eta)\right)(t)dx
-2(1-\delta)\int_s^t\int_\Omega \mu(\vartheta_\eta)
|D(\mathbf{u}_\eta)|^2\phi'(\vartheta_\eta)dxd\tau\\
&\quad-(1-\delta)\int_s^t\int_\Omega \lambda(\vartheta_\eta)
|{\rm div}\mathbf{u}_\eta|^2\phi'(\vartheta_\eta)dxd\tau+\int_s^t\int_\Omega \phi''(\vartheta_\eta)\kappa(\vartheta_\eta)|\nabla\vartheta_\eta|^2dxd\tau\\
&\leq \int_\Omega \left((\delta+\varrho_\eta)\phi(\vartheta_\eta)\right)(s)dx
-\delta\int_s^t\int_\Omega\vartheta_\eta^3 \phi'(\vartheta_\eta) dxd\tau\\
&\quad -\int_s^t\int_\Omega \vartheta_\eta p_\vartheta(\varrho_\eta){\rm div}\mathbf{u}_\eta\phi'(\vartheta_\eta)dxd\tau,
\end{split}
\end{equation}
where $D(\mathbf{u})=\frac{1}{2}\left(\nabla\mathbf{u}+\nabla^T\mathbf{u}\right)$.

Now, introducing a sequence of real numbers
\begin{equation}\label{6.9}
C_k=e^{-M[1-2^{-k}]}\quad \text{for all positive integers }  k,
\end{equation}
where $M$ is a positive number to be chosen later. Define $\phi_{k,\omega}$ by
\begin{equation}\label{6.10}
\phi_{k,\omega}(\vartheta_\eta)=\left[{\rm ln}\left(\frac{C_k}{\vartheta_\eta+\omega}\right)\right]_{+},
\end{equation}
it is easy to check that
\begin{equation}\label{6.11}
\phi_{k,\omega}'(\vartheta_\eta)=-\frac{1}{\vartheta_\eta+\omega} 1_{\{\vartheta_\eta+\omega\leq C_k\}},
\end{equation}
\begin{equation}\label{6.12}
\phi_{k,\omega}''(\vartheta_\eta)\geq \frac{1}{(\vartheta_\eta+\omega)^2} 1_{\{\vartheta_\eta+\omega\leq C_k\}}.
\end{equation}
Next define $U_{k,\,\omega}$ by
\begin{small}
\begin{equation}\label{6.13}
\begin{split}
U_{k,\,\omega}&:=\sup_{T_k\leq t\leq T}\left(\int_\Omega (\delta+\varrho_\eta) \phi_{k,\omega}(\vartheta_\eta)dx\right)\\
 &\quad\quad\quad +(1-\delta)\int_{T_k}^T\int_\Omega\frac{\nu(\vartheta_\eta)}{\vartheta_\eta+\omega}
1_{\{\vartheta_\eta+\omega\leq C_k\}}|D(\mathbf{u}_\eta)|^2 dxdt\\
&\quad\quad\quad +\int_{T_k}^T\int_\Omega\frac{\kappa(\vartheta_\eta)}{(\vartheta_\eta+\omega)^2}
1_{\{\vartheta_\eta+\omega\leq C_k\}}|\nabla\vartheta_\eta|^2 dxdt,
\end{split}
\end{equation}
\end{small}where $\{T_k\}$ is a sequence of non-negative numbers.
Note that $U_{k,\omega}$ depends on $\eta$, $\delta$ and $\omega$, that is, $U_{k,\,\omega}=U_{k, \eta, \delta, \omega}$, and for convenience, we
still write it as $U_{k,\omega}$.

If $T_k=0$ for all $k\in \mathbb{N}$, by \eqref{6.8} and \eqref{6.13}, then we claim that
\begin{equation}\label{6.14}
\begin{split}
U_{k,\omega}\leq &\int_\Omega \left(\delta+\varrho_{0,\delta}\right)\phi_{k,\omega}(\vartheta_{0,\delta})dx
+\delta\int_0^T\int_\Omega \frac{\vartheta_\eta^3}{\vartheta_\eta+\omega}1_{\{\vartheta_\eta+\omega\leq C_k\}}dxdt\\
&+\int_0^T\int_\Omega \frac{\vartheta_\eta}{\vartheta_\eta+\omega}1_{\{\vartheta_\eta+\omega\leq C_k\}}p_\vartheta(\varrho_\eta)|{\rm div}\mathbf{u}_\eta|dxdt.
\end{split}
\end{equation}
In fact, taking $0\leq T_{k-1}\leq s\leq T_{k}\leq t\leq T$ in \eqref{6.8} and combining \eqref{6.11}, \eqref{6.12} with \eqref{1.33}, we have
\begin{equation}\label{6.16}
\begin{split}
&\int_\Omega \left((\delta+\varrho_\eta) \phi_{k,\omega}(\vartheta_\eta)\right)(t)dx\\
&+(1-\delta)\int_{T_k}^t\int_\Omega\frac{\nu(\vartheta_\eta)}{\vartheta_\eta+\omega}
1_{\{\vartheta_\eta+\omega\leq C_k\}}|D(\mathbf{u}_\eta)|^2 dxd\tau\\
& +\int_{T_k}^t\int_\Omega\frac{\kappa(\vartheta_\eta)}{(\vartheta_\eta+\omega)^2}
1_{\{\vartheta_\eta+\omega\leq C_k\}}|\nabla\vartheta_\eta|^2 dxd\tau\\
&\leq \int_\Omega \left((\delta+\varrho_\eta) \phi_{k,\omega}(\vartheta_\eta)\right)(s)dx
+\delta\int_{T_{k-1}}^T\int_\Omega \frac{\vartheta_\eta^3}{\vartheta_\eta+\omega}1_{\{\vartheta_\eta+\omega\leq C_k\}}dxd\tau\\
&+\int_{T_{k-1}}^T\int_\Omega \frac{\vartheta_\eta}{\vartheta_\eta+\omega}1_{\{\vartheta_\eta+\omega\leq C_k\}}p_\vartheta(\varrho_\eta)|{\rm div}\mathbf{u}_\eta|dxdt.
\end{split}
\end{equation}
Taking the supremum over $t\in[T_k, T]$ on both sides of \eqref{6.16}, one deduces that
\begin{equation*}\label{6.17}
\begin{split}
U_{k,\omega} &\leq \int_\Omega \left((\delta+\varrho_\eta) \phi_{k,\omega}(\vartheta_\eta)\right)(s)dx
+\delta\int_{T_{k-1}}^T\int_\Omega \frac{\vartheta_\eta^3}{\vartheta_\eta+\omega}1_{\{\vartheta_\eta+\omega\leq C_k\}}dxdt\\
&+\int_{T_{k-1}}^T\int_\Omega \frac{\vartheta_\eta}{\vartheta_\eta+\omega}1_{\{\vartheta_\eta+\omega\leq C_k\}}p_\vartheta(\varrho_\eta)|{\rm div}\mathbf{u}_\eta|dxdt,
\end{split}
\end{equation*}
which implies \eqref{6.14} provided $T_k=0$ for all $k\in N$.

\underline{\bf{Step 2.}}
In this step, we prove that the second term on the right-hand side of \eqref{6.14} can be controlled by $U_{k-1,\omega}^{\sigma}$ for some $\sigma>1$. More precisely, we have
\begin{equation}\label{6.18}
\begin{split}
&\delta\int_0^T\int_\Omega \frac{\vartheta_\eta^3}{\vartheta_\eta+\omega}1_{\{\vartheta_\eta+\omega\leq C_k\}}dxdt
\leq C\frac{2^{k\alpha}}{M^{\alpha}}U_{k-1,\omega}^{\sigma},
\end{split}
\end{equation}
for some $\sigma>1$ and $\alpha>1$, where the constant $C$ is independent of $\eta,\delta>0$.

Indeed, if $\vartheta_\eta+\omega\leq C_k$, then for any $\omega>0$
\begin{equation}\label{6.19}
\frac{\vartheta_\eta^3}{\vartheta_\eta+\omega}\leq 1
\end{equation}
 by taking $M$ large enough such that $C_k$ is small enough, and
\begin{equation*}
\phi_{k-1,\omega}(\vartheta_\eta)=\left[{\rm ln}\left(\frac{C_{k-1}}{\vartheta_\eta+\omega}\right)\right]_{+}
\geq {\rm ln}\frac{C_{k-1}}{C_k},
\end{equation*}
which implies
\begin{equation}\label{6.21}
1_{\{\vartheta_\eta+\omega\leq C_k\}}\leq \left[{\rm ln}\frac{C_{k-1}}{C_k}\right]^{-\alpha}\phi_{k-1,\omega}(\vartheta_\eta)^{\alpha}, \,\,{\rm for \,\,any}\,\,\alpha>0.
\end{equation}
Combining \eqref{6.19} with \eqref{6.21}, we have
\begin{equation}\label{6.22}
\begin{split}
&\delta\int_0^T\int_\Omega \frac{\vartheta_\eta^3}{\vartheta_\eta+\omega}1_{\{\vartheta_\eta+\omega\leq C_k\}}dxdt\\
&\leq \delta^{1-\beta}\left[{\rm ln}\frac{C_{k-1}}{C_k}\right]^{-\alpha}\int_0^T\int_\Omega (\delta+\varrho_\eta)^\beta
\phi_{k-1,\omega}(\vartheta_\eta)^{\alpha}dxdt\\
&\leq C \delta^{1-\beta}\left[{\rm ln}\frac{C_{k-1}}{C_k}\right]^{-\alpha}T^{1/p'}|\Omega|^{1/q'}
\|(\delta+\varrho_\eta)^\beta\phi_{k-1,\omega}(\vartheta_\eta)^{\alpha}\|_{L^{p}(0,T; L^{q}(\Omega))},\\
\end{split}
\end{equation}
with $\frac{1}{p}+\frac{1}{p'}=1$ and $\frac{1}{q}+\frac{1}{q'}=1$. By Lemma \ref{6.20.} and the growth restriction imposed on $\kappa(\vartheta)$ \eqref{1.32}, the last term on the right-hand side of \eqref{6.22} can be controlled by
\begin{equation}\label{6.23}
\begin{split}
&\|(\delta+\varrho_\eta)^\beta\phi_{k-1,\omega}(\vartheta_\eta)^{\alpha}\|_{L^{p}(0,T; L^{q}(\Omega))}\\
&= \|\left((\delta+\varrho_\eta)\phi_{k-1,\omega}(\vartheta_\eta)\right)^{\beta/\alpha}
\phi_{k-1,\omega}(\vartheta_\eta)^{1-\beta/\alpha}\|_{L^{p \alpha}(0,T; L^{q \alpha}(\Omega))}^{\alpha}\\
&\leq \|\left((\delta+\varrho_\eta)\phi_{k-1,\omega}(\vartheta_\eta)\right)^{\beta/\alpha}\|
_{L^\infty(0,T;L^{\alpha/\beta}(\Omega))}^{\alpha}\\
&\quad\quad\quad\quad \|\phi_{k-1,\omega}(\vartheta_\eta)^{1-\beta/\alpha}\|_{L^{\frac{2}{1-\beta/\alpha}}(0,T; L^{\frac{6}{1-\beta/\alpha}}(\Omega))}^{\alpha}\\
&= \|(\delta+\varrho_\eta)\phi_{k-1,\omega}(\vartheta_\eta)\|^{\beta}_{L^\infty(0,T;L^1(\Omega))}
\|\phi_{k-1,\omega}(\vartheta_\eta)\|_{L^2(0,T; L^6(\Omega))}^{\alpha-\beta}\\
&\leq \|(\delta+\varrho_\eta)\phi_{k-1,\omega}(\vartheta_\eta)\|^{\beta}
_{L^\infty(0,T;L^1(\Omega))}\\
&\quad\quad\quad\quad     \left(\|\varrho_\eta \phi_{k-1,\omega}(\vartheta_\eta)\|_{L^\infty(0,T;L^1(\Omega))}
+\|\nabla\phi_{k-1,\omega}(\vartheta_\eta)\|_{L^2((0,T)\times\Omega)}\right)^{\alpha-\beta}\\
&\leq C U_{k-1,\omega}^{\beta}\left(U_{k-1,\omega}+U_{k-1,\omega}^{1/2}\right)^{\alpha-\beta}\\
&\leq C \left(U_{k-1,\omega}^{\alpha}+U_{k-1,\omega}^{\frac{\alpha+\beta}{2}}\right),
\end{split}
\end{equation}
with the coefficients $p$, $q$, $\alpha$ and $\beta$ satisfying
\begin{equation*}
\frac{1}{p \alpha}=\frac{1-\beta/\alpha}{2},\;
\frac{1}{q \alpha}=\frac{\beta}{\alpha}+\frac{1-\beta/\alpha}{6},
\end{equation*}
or equivalently,
\begin{equation*}
p=\frac{2}{\alpha-\beta},\;q=\frac{6}{\alpha+5\beta}.
\end{equation*}
Substituting \eqref{6.23} into \eqref{6.22}, we deduce
\begin{equation}\label{6.35}
\begin{split}
\delta\int_0^T\int_\Omega \frac{\vartheta_\eta^3}{\vartheta_\eta+\omega}1_{\{\vartheta_\eta+\omega\leq C_k\}}dxdt
\leq C \delta^{1-\beta}\left[{\rm ln}\frac{C_{k-1}}{C_k}\right]^{-\alpha}\left(U_{k-1,\omega}^{\alpha}
+U_{k-1,\omega}^{\frac{\alpha+\beta}{2}}\right).
\end{split}
\end{equation}
By \eqref{6.9}, we have
\begin{equation}\label{6.39}
\left[{\rm ln}\frac{C_{k-1}}{C_k}\right]^{-\alpha}=\frac{2^{k\alpha}}{M^{\alpha}}.
\end{equation}
Moreover, we can choose $\alpha>1$ and $0<\beta<1$ such that
\begin{equation}\label{6.400}
\sigma:=\min\left(\frac{\alpha+\beta}{2},\alpha\right)>1,
\end{equation}
and
\begin{equation}\label{6.401}
\delta^{1-\beta}\leq1.
\end{equation}
Combining \eqref{6.35}-\eqref{6.401} together, we obtain \eqref{6.18}.

\underline{\bf{Step 3.}}
In this step, we prove that the last term on the right-hand side of \eqref{6.14} can be controlled by $U_{k-1,\omega}^{\sigma}$ for the same $\sigma>1$ as Step 2.
To be precise, we have
\begin{equation}\label{6-01}
\begin{split}
&\int_0^T\int_\Omega \frac{\vartheta_\eta}{\vartheta_\eta+\omega}1_{\{\vartheta_\eta+\omega\leq C_k\}}p_\vartheta(\varrho_\eta)|{\rm div}\mathbf{u}_\eta|dxdt
\leq \frac12 U_{k,\omega}+C\frac{2^{k\alpha}}{M^{\alpha}}U_{k-1,\omega}^{\sigma},
\end{split}
\end{equation}
for the same $\sigma>1$ and $\alpha>1$ as Step 2, where the constant $C$ is independent of $\eta,\delta>0$.

In fact, by Cauchy-Schwarz and Young's inequality, we have
\begin{equation}\label{6-1}
\begin{split}
&\int_0^T\int_\Omega \frac{\vartheta_\eta}{\vartheta_\eta+\omega}1_{\{\vartheta_\eta+\omega\leq C_k\}}p_\vartheta(\varrho_\eta)|{\rm div}\mathbf{u}_\eta|dxdt\\
&\leq \frac18\int_0^T\int_\Omega \frac{\nu(\vartheta_\eta)}{\vartheta_\eta+\omega}1_{\{\vartheta_\eta+\omega\leq C_k\}}|D(\mathbf{u}_\eta)|^2 dxdt\\
&\quad+C\int_0^T\int_\Omega\frac{\vartheta_\eta}{\nu(\vartheta_\eta)}1_{\{\vartheta_\eta+\omega\leq C_k\}}|p_\vartheta(\varrho_\eta)|^2 dxdt\\
&\leq\frac12 U_{k,\omega}
+C\int_0^T\int_\Omega 1_{\{\vartheta_\eta+\omega\leq C_k\}}|p_\vartheta(\varrho_\eta)|^2 dxdt,
\end{split}
\end{equation}
where we have used \eqref{1.330} in the last inequality.
Then, by \eqref{6.21}, the last term on the right-hand side of \eqref{6-1} can be controlled by
\begin{equation}\label{6-2}
\begin{split}
&\int_0^T\int_\Omega 1_{\{\vartheta_\eta+\omega\leq C_k\}}|p_\vartheta(\varrho_\eta)|^2 dxdt\\
&\leq \left[{\rm ln}\frac{C_{k-1}}{C_k}\right]^{-\alpha}
\int_0^T\int_\Omega \phi_{k-1,\omega}(\vartheta_\eta)^{\alpha}(\delta+\varrho_\eta)^\beta
(\delta+\varrho_\eta)^{-\beta}|p_\vartheta(\varrho_\eta)|^2 dxdt\\
&\leq \left[{\rm ln}\frac{C_{k-1}}{C_k}\right]^{-\alpha}
\|(\delta+\varrho_\eta)^\beta\phi_{k-1,\omega}(\vartheta_\eta)^{\alpha}\|_{L^p(0,T;L^q(\Omega))}\\
&\quad\quad\quad\quad\quad\quad\quad \|(\delta+\varrho_\eta)^{-\beta}(p_\vartheta(\varrho_\eta))^2\|_{L^{p'}(0,T;L^{q'}(\Omega))}
\end{split}
\end{equation}
with $\frac{1}{p}+\frac{1}{p'}=1$ and $\frac{1}{q}+\frac{1}{q'}=1$,
where in accordance with \eqref{6.23}, the second term on the right-hand side can be controlled by
\begin{equation}\label{6-3}
\|(\delta+\varrho_\eta)^\beta\phi_{k-1,\omega}(\vartheta_\eta)^{\alpha}\|_{L^{p}(0,T; L^{q}(\Omega))}
\leq CU_{k-1,\omega}^{\sigma},\,{\rm with\,}\sigma>1,
\end{equation}
provided $\sigma=\min\{\frac{\alpha+\beta}{2}, \alpha\}>1$ and
the coefficients $p$, $q$, $\alpha$ and $\beta$ satisfy
\begin{equation*}
p=\frac{2}{\alpha-\beta},\quad q=\frac{6}{\alpha+5\beta}.
\end{equation*}
Thus, to complete the proof of \eqref{6-01}, it suffices to prove
\begin{equation}\label{6-4}
\|(\delta+\varrho_\eta)^{-\beta}(p_\vartheta(\varrho_\eta))^2\|_{L^{p'}(0,T;L^{q'}(\Omega))}\leq C.
\end{equation}
Note that $\frac{\alpha+\beta}{2}>1$, $\alpha>1$ and $0<\beta<1$ implies
\begin{equation*}
p<\frac{1}{1-\beta},\quad q<\frac{3}{1+2\beta},
\end{equation*}
which means
\begin{equation*}
p'>\frac{1}{\beta}, \quad q'>\frac{3}{2(1-\beta)}.
\end{equation*}
To prove \eqref{6-4}, it is enough to prove
\begin{equation}\label{6-5}
\begin{split}
&\|(\delta+\varrho_\eta)^{-\beta}(p_\vartheta(\varrho_\eta))^2\|_{L^{p'}(0,T;L^{q'}(\Omega))}\\
&\leq C \|(\delta+\varrho_\eta)^{-\beta}(1+\varrho^{\gamma/3})^2\|_{L^{p'}(0,T;L^{q'}(\Omega))}\\
&\leq C \|(\delta+\varrho_\eta)^{\frac{2\gamma}{3}-\beta}\|_{L^{p'}(0,T;L^{q'}(\Omega))}\\
&=C\|\delta+\varrho_\eta\|_{L^{(\frac{2\gamma}{3}-\beta) p'}(0,T;L^{(\frac{2\gamma}{3}-\beta) q'}(\Omega))}^{\frac{2\gamma}{3}-\beta}\\
&\leq C,
\end{split}
\end{equation}
which in accordance with $\varrho_\eta\in L^\infty(0,T; L^\gamma(\Omega))$, $\gamma>3/2$ can be achieved by choosing $\beta$ close enough to $0$.

\underline{\bf{Step 4.}}
In this step, we are ready to complete the proof of Proposition \ref{6.1.}.
By virtue of the assumption $\vartheta_{0,\delta}\geq\underline{\vartheta}>0$, we can choose $M$ large enough such that $e^{-M/2}<\underline{\vartheta}$, which implies for any $\omega>0$
\begin{equation}\label{6.37}
\phi_{k,\omega}(\vartheta_{0,\delta})=\left[{\ln}\left(\frac{e^{-M[1-2^{-k}]}}{\vartheta_{0,\delta}+\omega}\right)\right]_{+}=0.
\end{equation}
Substituting \eqref{6.18}, \eqref{6-01} and \eqref{6.37} into \eqref{6.14}, we obtain
\begin{equation}\label{6.402}
U_{k,\omega}\leq C\frac{2^{k\alpha}}{M^{\alpha}}U_{k-1,\omega}^{\sigma}\,\,{\rm with}\,\,\sigma>1.
\end{equation}
Therefore, by Lemma \ref{6.2.}, for M large enough (independently on $\eta$, $\delta$ and $\omega$), we deduce
\begin{equation}\label{6.43}
\lim_{k\rightarrow \infty} U_{k,\omega}=0,
\end{equation}
which with help of \eqref{6.11} and the definition of $U_{k,\omega}$ \eqref{6.13} yields
\begin{equation}\label{6.44}
\int_0^T\int_\Omega \kappa(\vartheta_\eta)\left|\nabla\left[\ln\left(\frac{e^{-M}}{\vartheta_\eta+\omega}\right)\right]_{+}\right|^2 dxdt=0,
\end{equation}
and
\begin{equation}\label{6.45}
\int_\Omega (\delta+\varrho_\eta)\left[\ln\left(\frac{e^{-M}}{\vartheta_\eta+\omega}\right)\right]_{+} dx=0.
\end{equation}
By the growth restriction imposed on $\kappa(\vartheta)$ \eqref{1.32}, \eqref{6.44} implies
\begin{equation}\label{6.46}
\left[\ln\left(\frac{e^{-M}}{\vartheta_\eta+\omega}\right)\right]_{+} \,\,{\rm is\,\,constant \,\,in\,\,\Omega\,\,for \,\,all\,\,}t\in[0,T].
\end{equation}
Furthermore, by \eqref{6.45}, we have
\begin{equation*}\label{6.47}
\left[\ln\left(\frac{e^{-M}}{\vartheta_\eta+\omega}\right)\right]_{+}=0,
\end{equation*}
This yields
\begin{equation*}\label{6.48}
\vartheta_\eta+\omega\geq e^{-M}
\end{equation*}
for any $\omega>0$, which completes our proof.
\end{proof}

In accordance with Proposition \ref{6.1.} and the strictly increasing property of $\mu(\vartheta)$, there exists a constant $\underline{\underline{\mu}}>0$ independent of the positive parameters $\eta$ and $\delta$ such that
\begin{equation*}
\mu(\vartheta_\eta)\geq\underline{\underline{\mu}}>0.
\end{equation*}
To conclude, we have the following proposition.

\begin{proposition}\label{4.1.}
For fixed $\delta>0$, under the hypotheses of Theorem \ref{1.2.},
the initial-boundary value problem \eqref{1.1}, \eqref{1.4} and \eqref{1.3} with the parameter $\delta>0$ admits an
approximate solution $(\varrho, \mathbf{u}, \vartheta)$, which is also the limit of the weak solution
constructed in Proposition \ref{3.1.}
when $\eta\to 0$, satisfying
\begin{enumerate}[(i)]
\item the density $\varrho\geq 0$ satisfies
\begin{equation*}
\varrho\in C([0,T];L^\beta_{weak}(\Omega))\cap L^{\beta+1}((0,T)\times\Omega)
\end{equation*}
and the initial condition \eqref{2-102}. The velocity $\mathbf{u}$ belongs to the space
$L^2(0,T; H^1_0(\Omega))$, and $(\varrho, \mathbf{u})$ solves the continuity equation $\eqref{1.1}_1$ in the sense of distributions;

\item $(\varrho, \mathbf{u}, \vartheta)$ solves a modified momentum equation
\begin{equation}\label{6-1000}
\partial_t(\varrho\mathbf{u})+{\rm div}(\varrho\mathbf{u}\otimes\mathbf{u})+\nabla(p(\varrho,\vartheta)+\delta\varrho^\beta)
={\rm div}\mathbb{S}
\end{equation}
in $\mathcal{D}'((0,T)\times\Omega)$, where the viscous stress tensor $\mathbb{S}$ is given by
\begin{equation*}
\mathbb{S}=\mu(\vartheta)(\nabla\mathbf{u}+\nabla^T\mathbf{u})+\lambda(\vartheta){\rm div}\mathbf{u}\mathbb{I}.
\end{equation*}
Moreover, $\varrho\mathbf{u}\in C([0,T]; L^{\frac{2\gamma}{\gamma+1}}_{weak}(\Omega))$ satisfies the initial condition \eqref{2-202};

\item the temperature $\vartheta\geq 0$ satisfies
\begin{equation*}
\vartheta \in L^3((0,T)\times\Omega),\quad \vartheta^{\frac{3-\omega}{2}}\in L^2(0,T;H^1(\Omega)),\,\,\omega\in(0,1),
\end{equation*}
and the initial condition \eqref{2-302} is satisfied in the sense of distributions.
Furthermore, the renormalized temperature inequality holds in $\mathcal{D}'((0,T)\times\Omega)$, that is,
\begin{equation}\label{6-110}
\begin{split}
&\int_0^T\int_\Omega (\delta+\varrho)H(\vartheta)\partial_t\varphi dxdt\\
&\quad+\int_0^T\int_\Omega \left(\varrho H(\vartheta)\mathbf{u}\cdot\nabla\varphi
+\mathcal{K}_h(\vartheta)\Delta\varphi-\delta\vartheta^3 h(\vartheta)\varphi\right) dxdt\\
&\leq\int_0^T\int_\Omega \left((\delta-1)\mathbb{S}:\nabla\mathbf{u}h(\vartheta)+h'(\vartheta)\kappa(\vartheta)
|\nabla\vartheta|^2\right)\varphi dxdt\\
&\quad+\int_0^T\int_\Omega h(\vartheta)\vartheta p_\vartheta(\varrho){\rm div}\mathbf{u}\varphi dxdt -\int_\Omega(\delta+\varrho_{0,\delta})H(\vartheta_{0,\delta})\varphi(0)dx
\end{split}
\end{equation}
for any $\varphi\in C_c^\infty([0,T]\times\Omega)$ satisfying
\begin{equation*}
\varphi\geq 0,\,\,\varphi(T,\cdot)=0,\,\,\nabla\varphi\cdot \mathbf{n}|_{\partial\Omega}=0,
\end{equation*}
where $H(\vartheta)=\int_0^\vartheta h(z) dz$ and $\mathcal{K}_h(\vartheta)=\int_0^\vartheta \kappa(z)h(z) dz$,
with the non-increasing $h\in C^2([0,\infty))$ satisfying
\begin{equation*}
0<h(0)<\infty,\,\,\lim_{z\rightarrow\infty}h(z)=0,\\
\end{equation*}
and
\begin{equation*}
h''(z)h(z)\geq 2(h'(z))^2\,\,for \,\,all\,\,z\geq0;
\end{equation*}

\item the energy inequality
\begin{equation}\label{6-120}
\begin{split}
&\int_0^T\int_\Omega(-\partial_t\psi)
\left(\frac{1}{2}\varrho|\mathbf{u}|^2+\varrho P_e(\varrho)+\frac{\delta}{\beta-1}\varrho^\beta
+(\delta+\varrho)\vartheta \right)dxdt\\
&\quad+\int_0^T\int_\Omega \psi \delta\left(\mathbb{S}:\nabla\mathbf{u}+ \vartheta^3\right) dxdt\\
&\leq\int_\Omega \left(\frac{1}{2}\frac{|\mathbf{m}_{0}|^2}{\varrho_{0,\delta}}
+\frac{\delta}{\beta-1}\varrho_{0,\delta}^\beta
+\varrho_{0,\delta} P_e(\varrho_{0,\delta})
+(\delta+\varrho_{0,\delta})\vartheta_{0,\delta} \right)dx
\end{split}
\end{equation}
holds for any $\psi\in C^\infty([0,T])$ satisfying
\begin{equation*}
\psi(0)=1,\quad \psi(T)=0,\quad \partial_t\psi\leq 0.
\end{equation*}
\end{enumerate}
\end{proposition}

\section{Passing to the limit for $\delta\to0$}\
\setcounter{equation}{0}
The final step is to deal with terms related to the parameter $\delta>0$ in \eqref{6-1000}-\eqref{6-120}. To this end, we denote by $(\varrho_\delta, \mathbf{u}_\delta, \vartheta_\delta)$ the weak solutions constructed in Proposition \ref{4.1.}.
Observe that estimates for $\varrho_\delta$ are similar to the previous sections, and estimates for $\mathbf{u}_\delta$ can be deduced after some calculations, thus our main task in this section is to deal with
terms related to $\vartheta_\delta$.

\subsection{\bf Estimates independent of $\delta>0$}\

For convenience, in the rest of this section, we denote $C$ a generic positive constant independent of $\delta>0$.

First, by \eqref{initial} and the energy inequality \eqref{6-120}, we have the following estimates
\begin{equation}\label{7-1}
\|\sqrt{\varrho_\delta}\mathbf{u}_\delta\|_{L^\infty(0,T; L^2(\Omega))}\leq C,
\end{equation}
\begin{equation}\label{7-2}
\|\varrho_\delta\|_{L^\infty(0,T; L^\gamma(\Omega))}\leq C,
\end{equation}
\begin{equation}\label{7-3}
\|(\delta+\varrho_\delta)\vartheta_\delta\|_{L^\infty(0,T; L^1(\Omega))}\leq C,
\end{equation}
\begin{equation}\label{7-4}
\delta \int_0^T\int_\Omega \mathbb{S}_\delta:\nabla\mathbf{u}_\delta dxdt\leq C,
\end{equation}
\begin{equation}\label{7-5}
\delta \int_0^T\int_\Omega \vartheta_\delta^3 dxdt\leq C.
\end{equation}

Then, taking $\varphi(t,x)=\psi(t)$ satisfying $0\leq\psi\leq1$, $\psi\in C_c^\infty(0,T)$ and $h(\vartheta)=\frac{\xi}{\xi+\vartheta}$ with $0<\xi<1$ in \eqref{6-110}, we have
\begin{equation}\label{7-6}
\begin{split}
&\int_0^T\int_\Omega \left(\frac{1-\delta}{\xi+\vartheta_\delta}\mathbb{S}_\delta:\nabla\mathbf{u}_\delta
+\frac{\kappa(\vartheta_\delta)}{(\xi+\vartheta_\delta)^2}|\nabla\vartheta_\delta|^2\right)\psi
+\varrho_\delta \ln(\xi+\vartheta_\delta)\partial_t\psi dxdt\\
&\leq \delta\int_0^T\int_\Omega \frac{\vartheta_\delta^{3}}{\xi+\vartheta_\delta}\psi dxdt
+\int_0^T\int_\Omega \frac{\vartheta_\delta}{\xi+\vartheta_\delta}p_\vartheta(\varrho_\delta){\rm div}\mathbf{u}_\delta \psi dxdt.
\end{split}
\end{equation}
By virtue of \eqref{7-5}, we take the limit for $\xi\to 0$ in \eqref{7-6} to deduce
\begin{equation}\label{7-7}
\begin{split}
&\int_0^T\int_\Omega \left(\frac{1-\delta}{\vartheta_\delta}\mathbb{S}_\delta:\nabla\mathbf{u}_\delta
+\frac{\kappa(\vartheta_\delta)}{\vartheta_\delta^2}|\nabla\vartheta_\delta|^2\right)\psi
+\varrho_\delta \ln\vartheta_\delta\partial_t\psi dxdt\\
&\leq C\left(1+\int_0^T\int_\Omega p_\vartheta(\varrho_\delta){\rm div}\mathbf{u}_\delta \psi dxdt\right),
\end{split}
\end{equation}
where by the continuity equation $\eqref{1.1}_1$, the last term on the right-hand side can be rewritten as
\begin{equation*}
\int_0^T\int_\Omega p_\vartheta(\varrho_\delta){\rm div}\mathbf{u}_\delta\psi dxdt
=\int_0^T\int_\Omega \varrho_\delta P_\vartheta(\varrho_\delta) \partial_t\psi dxdt,
\end{equation*}
with
\begin{equation*}
P_\vartheta(\varrho)=\int_1^\varrho \frac{p_\vartheta(z)}{z^2} dz.
\end{equation*}
Thus, by assumption \eqref{1.31}, the growth restriction imposed on $\kappa(\vartheta)$ \eqref{1.32} and estimate \eqref{7-2}, we have
\begin{equation*}
\|\varrho_\delta \ln \vartheta_\delta\|_{L^\infty(0,T; L^1(\Omega))}\leq C,
\end{equation*}
\begin{equation*}
\|\nabla \ln \vartheta_\delta\|_{L^2((0,T)\times\Omega)}\leq C,
\end{equation*}
\begin{equation*}
\|\nabla\vartheta_\delta\|_{L^2((0,T)\times\Omega)}\leq C,
\end{equation*}
which combined with Lemma \ref{6.20.} implies
\begin{equation}\label{7-8}
\|\ln \vartheta_\delta\|_{L^2(0,T;H^1(\Omega))}\leq C,
\end{equation}
\begin{equation}\label{7-9}
\|\vartheta_\delta\|_{L^2(0,T;H^1(\Omega))}\leq C.
\end{equation}

Next, taking $h(\vartheta)=\frac{1}{(1+\vartheta)^l}$ with $0<l<1$ in \eqref{6-110}, we obtain
\begin{equation}\label{7-10}
\begin{split}
&\int_0^T\int_\Omega \left(\frac{1-\delta}{(1+\vartheta_\delta)^l}\mathbb{S}_\delta:\nabla\mathbf{u}_\delta
+l\frac{\kappa(\vartheta_\delta)}{(1+\vartheta_\delta)^{l+1}}|\nabla\vartheta_\delta|^2\right)\psi dxdt\\
&\leq \delta\int_0^T\int_\Omega \frac{\vartheta_\delta^{3}}{(1+\vartheta_\delta)^l}\psi dxdt
-\int_0^T\int_\Omega (\delta+\varrho_{\delta})H(\vartheta_{\delta})\partial_t\psi dxdt\\
&+\int_0^T\int_\Omega \frac{\vartheta_\delta}{(1+\vartheta_\delta)^l}p_\vartheta(\varrho_\delta){\rm div}\mathbf{u}_\delta \varphi dxdt,
\end{split}
\end{equation}
with $H(\vartheta)=\int_0^\vartheta \frac{1}{(1+z)^l} dz$.
Letting $l\rightarrow 0$ in \eqref{7-10} and combining with estimates \eqref{7-2}-\eqref{7-5} and \eqref{7-9}, we have
\begin{equation*}
\int_0^T\int_\Omega\mathbb{S}_\delta:\nabla\mathbf{u}_\delta dxdt\leq C,
\end{equation*}
which with help of Proposition \ref{6.1.} and assumptions imposed on $\mu(\vartheta)$ and $\lambda(\vartheta)$ in Theorem \ref{1.2.}, yields
\begin{equation}\label{7-11}
\|\mathbf{u}_\delta\|_{L^2(0,T; H^{1}_0(\Omega))}\leq C.
\end{equation}

Moreover, for fixed $0<l<1$ in \eqref{7-10}, by the growth restriction imposed on $\kappa(\vartheta)$ \eqref{1.32}, we obtain
\begin{equation}\label{7-12}
\|\vartheta_\delta^{\frac{3-l}{2}}\|_{L^2(0,T;H^1(\Omega))}\leq C(l),
\end{equation}
with the constant $C(l)$ depending on $l\in(0,1)$. Combining estimate \eqref{7-12} with \eqref{7-3} and thanks to the interpolation inequality, we deduce for a certain $p>1$ and a small positive number $\omega$
\begin{equation}\label{7-13}
\vartheta_\delta^3\,\,{\rm is \,\,bounded \,\,in}\,\,L^p\left(\{\varrho_\delta(t,x)\geq\omega>0\}\right)
\end{equation}
by a positive constant independent of $\delta>0$.

\subsection{\bf Strong convergence of the temperature $\vartheta_\delta$}\

Following Chapter 7 in \cite{Feireisl Dynamics 2004} and Chapter 5 in \cite{Feireisl On the motion 2004}, to obtain the strong convergence of the temperature
\begin{equation}\label{10-1}
\vartheta_\delta \to \vartheta\quad {\rm in}\,\,L^2(\{\varrho>0\}),
\end{equation}
by estimate \eqref{7-9}, it suffices to show that
\begin{equation}\label{10-2}
\varrho_\delta H(\vartheta_\delta) \to \varrho \overline{H(\vartheta)}\quad {\rm in\,\,}L^2(0,T; W^{-1,2}(\Omega)).
\end{equation}

First, we introduce the following lemma, which can be regarded as a variant of the celebrated Aubin-Lions lemma (see Lemma 6.3 in \cite{Feireisl Dynamics 2004}).
\begin{lemma}\label{7.1.}
Let $\{\mathbf{v}_n\}_{n=1}^\infty$ be a sequence of functions such that
\begin{equation*}
\mathbf{v}_n {\rm\,\,is \,\,bounded \,\,in}\,\,L^2(0,T; L^q(\Omega))\cap L^\infty(0,T; L^1(\Omega)), \quad{\rm with} \,\,q>6/5,
\end{equation*}
furthermore, assume that
\begin{equation*}
\partial_t{\mathbf{v}_n} \geq l_n \,\,{\rm in} \,\,\mathcal{D}'((0,T)\times\Omega),
\end{equation*}
where
\begin{equation*}
l_n {\rm\,\,is\,\,bounded \,\, in}\,\,L^1(0,T; W^{-m,r}(\Omega))
\end{equation*}
for a certain $m\geq1$, $r>1$.

Then $\{\mathbf{v}_n\}_{n=1}^\infty$ contains a subsequence such that
\begin{equation*}
\mathbf{v}_n\rightarrow \mathbf{v}\,\,{\rm in} \,\,L^2(0,T; H^{-1}(\Omega)).
\end{equation*}
\end{lemma}

Now we apply Lemma \ref{7.1.} to the sequence $(\delta+\varrho_\delta) H(\vartheta_\delta)$. By the temperature inequality \eqref{6-110}, to obtain \eqref{10-2}, it is enough to prove that
\begin{equation*}
\|\vartheta_\delta\|_{L^3((0,T)\times\Omega)}\leq C,
\end{equation*}
which by \eqref{7-13} can be achieved provided that
\begin{equation}\label{7-15}
\vartheta_\delta\,\,{\rm is \,\,bounded \,\,in \,\,}L^{3}\left(\{\varrho_\delta(x,t)<\omega\}\right)
\end{equation}
by a positive constant independent of $\delta>0$, with $\omega$ being a sufficiently small positive number.

As proved in \cite{Feireisl Dynamics 2004, Feireisl On the motion 2004}, the estimate \eqref{7-15} can be obtained by choosing the function
\begin{equation*}
\varphi(t,x)=\psi(t)(\eta(t,x)-\underline{\eta}),\,\,0\leq\psi\leq1,\,\,\psi\in C_c^\infty(0,T),
\end{equation*}
where
\[\underline{\eta}=\inf_{t\in[0,T],x\in\Omega}\eta,\]
and for each $t\in[0,T]$, $\eta=\eta_\delta$ is the unique solution of the following Neumann problem
\begin{equation*}
\begin{cases}
\Delta\eta_\delta(t)=B(\varrho_\delta(t))-\frac{1}{|\Omega|}\int_\Omega B(\varrho_\delta(t))dx\,\,{\rm in}\,\,\Omega,\\
\
\nabla\varrho_\delta\cdot\mathbf{n}=0\,\,{\rm on}\,\,\partial\Omega,\\
\
\int_\Omega\eta_\delta(t)dx=0,
\end{cases}
\end{equation*}
with $B\in C^\infty(\mathbb{R})$ non-increasing satisfying
\begin{equation*}
B(z)=
\begin{cases}
0, & {\rm if}\,\, z\leq \omega,\\
-1, & {\rm if}\,\, z\geq 2\omega,
\end{cases}
\end{equation*}
as a test function of the renormalized temperature inequality \eqref{6-110}.

\subsection{\bf Strong convergence of the density $\varrho_\delta$}\

In this subsection, our main goal is to prove
\begin{equation}\label{7-18}
\varrho_\delta\to \varrho \quad{\rm in}\,\,L^1((0,T)\times\Omega).
\end{equation}
As in \cite{Feireisl On the motion 2004}, this can be achieved by taking the quantity
\begin{equation*}
\varphi(t,x)=\psi(t)\eta(x)\Delta^{-1}\partial_{x_i}[T_k(\varrho_\delta)]
\end{equation*}
with $\psi\in C_c^{\infty}(0,T)$, $\eta \in C_c^\infty(\Omega)$, and $T_k(\varrho)$ being cut-off functions
\begin{equation*}
T_k(\varrho)=\min\{\varrho,k\},\quad k\geq1,
\end{equation*}
as a test function of the momentum equation \eqref{6-1000}.

\subsection{\bf Passing the limit}\

Taking into account \eqref{7-11},\eqref{10-1} and \eqref{7-18}, we can pass the limit $\delta\to 0$ for the approximate solutions $(\varrho_\delta, \mathbf{u}_\delta, \vartheta_\delta)$ constructed in Proposition \ref{4.1.}, thus Theorem \ref{1.2.} is proved.

\bigskip
\noindent{\bf Acknowledgments:}
 Guodong Wang was supported by National Natural Science Foundation of China (12001135 and 12071098) and
China Postdoctoral Science Foundation (2019M661261).

\end{document}